\documentclass[11pt]{article}
\usepackage[utf8]{inputenc}
\usepackage{graphicx}
\usepackage{latexsym}
\usepackage{amsmath}
\usepackage{amssymb}
\usepackage{amsthm}
\usepackage{hyperref}
\usepackage{multicol} 
\usepackage{multirow}
\usepackage{caption}
\usepackage{MnSymbol}
\usepackage{mathtools}
\usepackage{pgf,tikz}
\usepackage{tikz,fancyhdr}
\usepackage{tkz-graph}
\usepackage{array}
\tikzstyle{vertex}=[circle,draw, inner sep=0pt, minimum size=5pt] 
\newcommand{\vertex}{\node[vertex]}
\usetikzlibrary{decorations.markings}
\usepackage{lipsum}
\usetikzlibrary{patterns,arrows,decorations.pathreplacing}
\usepackage{pst-node,pst-plot}

\allowdisplaybreaks
\def \be {{\bemma}}

\def \L {\mathfrak{L}}
\def \J {\mathbb{J}}

\def \D {\mathfrak{D}}

\def \I {\mathbb{I}}
\def \S {\mathfrak{S}}

\def \x {{\lambda_1}}
\def \y {{\lambda_2}}
\def \z {{\lambda_3}}
\def \n {{\eta}}
\def \u {{\mu}}
\def \be {{\beta_m}}

\newtheorem*{theorem*}{Theorem}
\newtheorem{theorem}{Theorem}
\newtheorem{re}[theorem]{Remark}
\newtheorem{cor}[theorem]{Corollary}

\newtheorem{lemma}[theorem]{Lemma}
\newtheorem{ex}[theorem]{Example}

\newtheorem{definition}[theorem]{Definition}

\newtheorem{pro}[theorem]{Proposition}

\title{Algebraic connectivity of Kronecker products of line graphs}
\author{Shivani Chauhan,  A. Satyanarayana Reddy\\\\
Department of 
Mathematics, Shiv Nadar 
Institution of Eminence, India-201314\\ (e-mail: 
sc739@snu.edu.in, satya.a@snu.edu.in)}
\begin{document}
\maketitle
\begin{abstract}
Let $X$ be a tree with $n$ vertices and $L(X)$ be its line graph. In this work, we completely characterize the trees for which the algebraic connectivity of $L(X)\times K_m$ is equal to $m-1$, where $\times$ denotes the Kronecker product. We provide a few necessary and sufficient conditions for $L(X)\times K_m$ to be Laplacian integral. The algebraic connectivity of $L(X)\times K_m$, where $X$ is a tree of diameter 4 and $k$-book graph is discussed.
\end{abstract}


\section{Introduction}
The {\em Laplacian matrix} $\L(X)$ of a graph $X$ is defined as $\D(X)-A(X)$, where $A(X)$ is the {\em adjacency matrix} and $\D(X)$ is the diagonal matrix with diagonal entries as the row sums of $A(X).$ It is well known that $\L(X)$ is singular and positive semidefinite. Fiedler \cite{fiedler1973algebraic}, showed that the second smallest eigenvalue of $\L(X)$ is zero if and only if $X$ is disconnected. This eigenvalue is termed as the {\em algebraic connectivity} of graph $X$, denoted by $a(X)$, and the corresponding eigenvector is known as the {\em Fiedler vector}. The collection of all the eigenvalues of a matrix together with its multiplicities is known as the {\em spectrum} of that matrix. The spectrum of the Laplacian matrix is called {\em Laplacian spectrum}.
For more information on Laplacian matrices and algebraic connectivity one can refer to the survey article of Merris \cite{merris1994laplacian} and Maria \cite{de2007old}, respectively. The {\em path graph, complete graph} and the {\em star graph} on $n$ vertices are denoted by $P_n, K_n$ and $K_{1,n-1}$, respectively. Grone {\em et al.}\cite{grone1990laplacian} proved that if $X\neq K_{1,n-1}$ is a tree with $n\geq 6$ vertices, then $a(X)<0.49.$ Also, if $X\neq K_2$ is a tree, then $a(X)\in(0,1]$, and it is 1 if and only if $X=K_{1,n-1}.$
\par
The {\em Kronecker product} $X_ 1\times X_2$ of graphs $X_1$ and $X_2$ is a graph such that the vertex set is $V(X_1)\times V(X_2)$, vertices $(x_1,x_2)$ and $(x_1^\prime,x_2^\prime)$ are adjacent if $x_1$ is adjacent to $x_1^\prime$ and $x_2$ is adjacent to $x_2^\prime.$ The adjacency matrix of $X_1\times X_2$ is $A(X_1)\otimes A(X_2)$, where $A\otimes B$ denotes the Kronecker product of  matrices $A$ and $B.$ We use $\I_n$ and $\J_n$ to denote the identity matrix and the matrix with all entries 1 of order $n$, respectively, and  $\bold{0_{m\times n}}$ denotes the null matrix.
It is easy to see that
$$\L(X\times K_m)
=\J_m\otimes(-A(X)) +\I_m\otimes(A(X)+(m-1)\D(X)).$$
The structure of $\L(X\times K_m)$ is given by
\begin{equation*}
\L(X\times K_m)=\begin{bmatrix}
(m-1)\D(X) & -A(X) & -A(X) & \cdots & -A(X)\\
-A(X) & (m-1)\D(X) & -A(X) & \cdots & -A(X)\\
-A(X) & -A(X) & (m-1)\D(X) & \cdots & -A(X)\\
\vdots & \vdots & \vdots & \ddots & \vdots\\
-A(X) & -A(X) & -A(X) & \cdots & (m-1)\D(X)
\end{bmatrix}.
\end{equation*}
The spectrum of $\L(X\times K_m)$ can be obtained from the spectrum of $(m-1)\L(X)$ and $A(X)+(m-1)\D(X).$
Let $\lambda$ be an eigenvalue of $\L(X)$ with its corresponding eigenvector $v$, then for $V=(v,v,\ldots,v)^T$, $\L(X\times K_m)V=(m-1)\lambda V.$ Suppose that $\lambda^\prime$ be an eigenvalue of $A(X)+(m-1)\D(X)$ with corresponding eigenvector $v^\prime.$ Then  $$\{(v^\prime,-v^\prime,\bold{0_{n\times 1}},\bold{0_{n\times 1}},\ldots,\bold{0_{n\times 1}})^T,
(v^\prime,\bold{0_{n\times 1}},-v^\prime,\bold{0_{n\times 1}},\ldots,\bold{0_{n\times 1}})^T,\ldots,(v^\prime, \bold{0_{n\times 1}},\bold{0_{n\times 1}},\ldots,-v^\prime)^T
\}$$ are $m-1$ dimensional linearly independent eigenvectors of $\L(X\times K_m).$ As the sum of the multiplicities of the obtained eigenvalues equals the number of vertices in $X\times K_m$, it shows that the union of the spectrum of  $A(X)+(m-1)\D(X)$ and $(m-1)\L(X)$ with multiplicity $m-1$ and 1 respectively, yields the spectrum of $\L(X\times K_m).$ For the sake of convenience, we denote the matrix $A(X)+(m-1)\D(X)$ by $Q_{m-1}(X)$ or simply $Q_{m-1}.$ One can see that $Q_1(X)$ is the {\em signless Laplacian matrix} of the graph $X.$ For more information about $Q_1$ one can refer  \cite{cvetkovic2009towards,cvetkovic2010towards2,cvetkovic2010towards3}.
We illustrate the above process by computing the spectrum of $K_{n-1}\times K_m$ in Example \ref{ex:star}.
\par
\begin{ex}\label{ex:star}
The eigenvalues of $\L(K_{n-1}\times K_m)$ are given by the union of spectra of $\J_{n-1}+((m-1)(n-2)-1)\I_{n-1}$ and $(m-1)(-\J_{n-1}+(n-1)\I_{n-1}).$ The eigenvalues of $\J_n$ are $0$ and $n$ with multiplicity $n-1$ and $1,$ respectively. It is easy to check that the eigenvalues of $\L(K_{n-1}\times K_m)$ are $0,(m-1)(n-2)-1,n+(m-1)(n-2)-2$ and $(n-1)(m-1).$ Therefore, $a(K_{n-1}\times K_m)=(n-2)(m-1)-1.$
\end{ex}
\par
If $X$ is a tree, we denote the graph $L(X)\times K_m$ by $\be(X)$, where $L(X)$ denotes the line graph of $X.$ The objective of this paper is to compute $a(\be(X)).$ Note that, $a(\be(X))=min\{(m-1)a(L(X)),q_{min}^{m-1}(L(X))\},$ where $q_{min}^{m-1}(Y)$ denotes the smallest eigenvalue of $Q_{m-1}(Y).$ Hence, the problem boils down to computing $a(L(X))$ and $q_{min}^{m-1}(L(X)).$ From Example \ref{ex:star}, $a(\be(K_{1,n-1}))\linebreak=(n-2)(m-1)-1$ as $L(K_{1,n-1})=K_{n-1}.$ The converse is also true {\it i.e.,} if
$a(\be(X))=(n-2)(m-1)-1$, then $X=K_{1,n-1}.$ It is well known that the line graph of a tree has a cut vertex and $a(X)\leq \kappa (X)$, where $X\neq K_{n}$ and $\kappa(X)$ denotes the vertex connectivity of $X.$ This implies $a(L(X))\leq 1.$ The spectrum of $(m-1)\L(L(X))$ is contained in the spectrum of $\L(\be(X))$, therefore $0\leq a(\be(X))\leq m-1.$ Thus one can see that there is a large gap between $a(\be(K_{1,n}))$ and $a(\be(X)),$ where $X$ is different from $K_{1,n-1}$. 
\par
We know that $a(\be(X))\geq 0.$ By the result of Imrich and Kla\v{v}zar \cite{imrich2000s}, which states that $X_1\times X_2$ is connected if and only if both $X_1$ and $X_2$ are connected and at least one of them is non-bipartite, we deduce that $a(\be(X))=0$ if and only if $X=P_n$ and $m=2.$ So we define, 
$$\S_m=\{X \mid \mbox{$X$ is a tree and}\; 0<a(\be(X))\leq m-1\}.$$ In Figure \ref{fig:Structure1}, we demonstrate the structure of $X, L(X),\beta_3(X),\beta_4(X)$ and $\beta_5(X)$, where $X=P_4.$ More examples of $a(\be(X))$ can be found in Table \ref{tab:algebraic}.

\begin{figure}[h]
\centering
\begin{tabular}{ccccc}
     \begin{tikzpicture}[scale=0.5]
    \vertex (1) at (0,1){};
    \vertex (2) at (1,1) {};
    \vertex (3) at (2,1){};
    \vertex (4) at (3,1){};
    \path[-]
    (1) edge (2)
    (2) edge (3)
    (3) edge (4)
    ;
\end{tikzpicture}
&
\begin{tikzpicture}[scale=0.5]
    \vertex (1) at (0,1){};
    \vertex (2) at (1,1){};
    \vertex (3) at (2,1){};
    
    \path[-]
    (1) edge (2)
    (2) edge (3)
    ;
\end{tikzpicture}
&
\begin{tikzpicture}
    \vertex (1) at (0,0){};
    \vertex (2) at (0,1) {};
    \vertex (3) at (0,2) {};
    \vertex (4) at (1,0){};
    \vertex (5) at (1,1) {};
    \vertex (6) at (1,2) {};
    \vertex (7) at (2,0) {};
    \vertex (8) at (2,1) {};
    \vertex (9) at (2,2) {};
    \path[-]
    (4) edge (2)
    (7) edge (2)
    (1) edge (5)
    (7) edge (5)
    (8) edge (1)
    (8) edge (4)
    (3) edge (5)
    (3) edge (8)
    (6) edge (2)
    (6) edge (8)
    (9) edge (2)
    (9) edge (5)
    ;
\end{tikzpicture}
 & 
 \begin{tikzpicture}
    \vertex (1) at (0,0){};
    \vertex (2) at (0,1) {};
    \vertex (3) at (0,2) {};
    \vertex (4) at (1,0){};
    \vertex (5) at (1,1) {};
    \vertex (6) at (1,2) {};
    \vertex (7) at (2,0) {};
    \vertex (8) at (2,1) {};
    \vertex (9) at (2,2) {};
    \vertex (10) at (3,0) {};
    \vertex (11) at (3,1) {};
    \vertex (12) at (3,2) {};
    \path[-]
    (4) edge (2)
    (7) edge (2)
    (1) edge (5)
    (7) edge (5)
    (8) edge (1)
    (8) edge (4)
    (3) edge (5)
    (3) edge (8)
    (6) edge (2)
    (6) edge (8)
    (9) edge (2)
    (9) edge (5)
    (3) edge (11)
    (12) edge (2)
    (12) edge (5)
    (12) edge (8)
    (11) edge (9)
    (6) edge (11)
    (11) edge (1)
    (11) edge (4)
    (11) edge (7)
    (2) edge (10)
    (10) edge (5)
    (8) edge (10)
    ;
\end{tikzpicture}
& 
\begin{tikzpicture}
    \vertex (1) at (0,0){};
    \vertex (2) at (0,1){};
    \vertex (3) at (0,2){};
    \vertex (4) at (1,0){};
    \vertex (5) at (1,1){};
    \vertex (6) at (1,2){};
    \vertex (7) at (2,0){};
    \vertex (8) at (2,1){};
    \vertex (9) at (2,2){};
    \vertex (10) at (3,0){};
    \vertex (11) at (3,1){};
    \vertex (12) at (3,2){};
    \vertex (13) at (4,0){};
    \vertex (14) at (4,1){};
    \vertex (15) at (4,2){};
    \path[-]
    (4) edge (2)
    (7) edge (2)
    (1) edge (5)
    (7) edge (5)
    (8) edge (1)
    (8) edge (4)
    (3) edge (5)
    (3) edge (8)
    (6) edge (2)
    (6) edge (8)
    (9) edge (2)
    (9) edge (5)
    (3) edge (11)
    (12) edge (2)
    (12) edge (5)
    (12) edge (8)
    (11) edge (9)
    (6) edge (11)
    (11) edge (1)
    (11) edge (4)
    (11) edge (7)
    (2) edge (10)
    (10) edge (5)
    (8) edge (10)
    (15) edge (2)
    (15) edge (5)
    (15) edge (8)
    (15) edge (11)
    (14) edge (3)
    (14) edge (6)
    (14) edge (9)
    (14) edge (12)
    (13) edge (2)
    (13) edge (5)
    (13) edge (8)
    (13) edge (11)
    (14) edge (1)
    (14) edge (4)
    (14) edge (7)
    (14) edge (10)
    ;
\end{tikzpicture}
\\
\end{tabular}
\caption{$P_4,L(P_4),\beta_3(P_4),\beta_4(P_4)$ and $\beta_5(P_4)$ }
\label{fig:Structure1}
\end{figure}
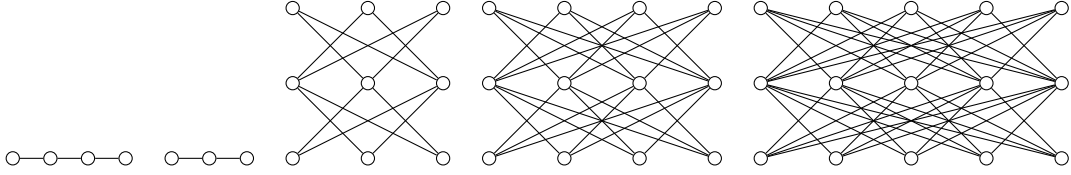
\par
This paper is organized as follows: In Section \ref{sec:2}, we explore all $X\in \S_m$ for which $a(\be(X))=m-1$, and necessary and sufficient conditions for $\be(X)$ to be Laplacian integral is given. In Section \ref{sec:3}, we compute the algebraic connectivity of $L(X)\times K_m,$ where $X$ is a tree of diameter 4 and $k$-book graph.
%%%%%%
\section{Main results}\label{sec:2}
 Let $T(k,s,t)$ denotes the tree with $n=s+t+k+1$ vertices, constructed by taking a path on vertices $\{1,2,\ldots,k+1\}$, adding $s$ and $t$ pendant vertices adjacent to vertex $1$ and $k+1$, respectively. In Figure \ref{fig:T(2,2,3)}, $T(2,3,2)$ is given.
 In this section, it is proved that if $X\in \S_m$, then $a(\be(X))=m-1$ if and only if $X$ is a $T(1,s,t)$, where $s,t\geq 2$. As a consequence, we find all $X$ for which $\be(X)$ is the Laplacian integral. We observed that $L(T(k,s,t))$ is a restricted graph. Hence, we defined a restricted graph and classified all restricted graphs $X$ for which $a(X\times K_m)=m-1.$ We mention a few results that will be used later. 
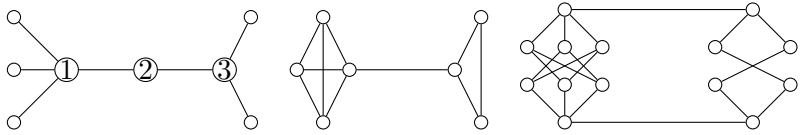
\begin{figure}[h]
\centering
\begin{tabular}{ccc}
     \begin{tikzpicture}[scale=0.7]
    \vertex (-1) at (-1,0){};
    \vertex (1) at (0,0) {1};
    \vertex (2) at (1.5,0){2};
    \vertex (3) at (-1,-1){};
    \vertex (4) at (-1,1) {};
    \vertex (5) at (3,0) {3};
    \vertex (6) at (3.5,-1){};
    \vertex (7) at (3.5,1){};
    \path[-]
    (1) edge (2)
    (1) edge (3)
    (-1) edge (1)
    (1) edge (4)
    (2) edge (5)
    (5) edge (6)
    (5) edge (7)
    ;
\end{tikzpicture}
&
\begin{tikzpicture}[scale=0.7]
    \vertex (1) at (0,0){};
    \vertex (2) at (1,0){};
    \vertex (3) at (0.5,-1){};
    \vertex (4) at (0.5,1){};
    \vertex (5) at (3,0){};
    \vertex (6) at (3.5,-1){};
    \vertex (7) at (3.5,1){};
   
    \path[-]
    (1) edge (2)
    (1) edge (3)
    (1) edge (4)
    (2) edge (5)
    (5) edge (6)
    (5) edge (7)
    (2) edge (4)
    (2) edge (3)
    (3) edge (4)
    (6) edge (7)
    ;
\end{tikzpicture}
&
\begin{tikzpicture}[yscale=0.5]
    \vertex (1) at (1.5,2){};
    \vertex (2) at (4,2) {};
    \vertex (3) at (1,1) {};
    \vertex (4) at (1.5,1){};
    \vertex (5) at (2,1) {};
    \vertex (6) at (3.5,1) {};
    \vertex (7) at (4.5,1) {};
    \vertex (9) at (1.5,-1) {};
    \vertex (10) at (4,-1) {};
    \vertex (11) at (1,0) {};
    \vertex (12) at (1.5,0){};
    \vertex (13) at (2,0) {};
    \vertex (14) at (3.5,0) {};
    \vertex (16) at (4.5,0) {};
    \path[-]
    (1) edge (2)
    (1) edge (3)
    (1) edge (4)
    (1) edge (5)
    (2) edge (6)
    (2) edge (7)
    (9) edge (10)
    (9) edge (11)
    (9) edge (13)
    (9) edge (12)
    (10) edge (14)
    (10) edge (16)
    (3) edge (12)
    (3) edge (13)
    (4) edge (11)
    (4) edge (13)
    (5) edge (11)
    (5) edge (12)
    (6) edge (16)
    (7) edge (14)
    ;
\end{tikzpicture}
\\
\end{tabular}
\caption{$T(2,3,2),L(T(2,3,2))$ and $\beta_2(T(2,3,2))$}
\label{fig:T(2,2,3)}
\end{figure}
\par

\begin{pro}
\cite[Corollary 2.1]{kirkland2000bound}\label{thm:cutpoint}
Let $X$ be a connected graph with a cut vertex $v.$ Then $a(X)=1$ if and only if $v$ is  adjacent to every vertex of $X$.
\end{pro}

\begin{theorem}\cite[Theorem 3.3]{das2004laplacian}\label{thm:das}
Let $X=(V,E)$ be a graph with $N_X(v_1)=N_X(v_2)=\cdots=N_X(v_k)=\{w_{1},w_{2},\ldots, w_{p}\}$, where $N_X(v)$ denotes the neighborhood of vertex $v$ in $X$. If $X^{+}$ is the graph obtained from $X$ by adding any $q$ $\left(1\leq q\leq \frac{k(k-1)}{2}\right)$ edges among $\{v_1,v_2,\ldots,v_k\}$, 
then the eigenvalues of $\L(X^{+})$ are as follows: Laplacian eigenvalues of the graph $X$ which are equal to $p$(k-1 in number) are incremented by $\lambda_i(X^+[G])$, $i=1,2,\ldots,k-1$ and,
the remaining eigenvalues are same, where $X^{+}[G]$ denotes the induced subgraph of $X^{+}$ with vertex set $G=\{v_1,v_2,\ldots,v_k\}$. 
\end{theorem}

\begin{ex} The line graph of $T(1,s,t)$ has exactly one cut vertex which is adjacent to every vertex of $L(T(1,s,t)).$ In Figure \ref{fig:T(1,2,3)}, $v$ is the cut vertex of $L(T(1,3,2)).$
\begin{figure}[h]
\centering
\begin{tabular}{cc}
     \begin{tikzpicture}[scale=0.7]
    \vertex (-1) at (-1,0){};
    \vertex (1) at (0,0){};
    \vertex (3) at (-1,-1){};
    \vertex (4) at (-1,1){};
    \vertex (5) at (2,0) {};
    \vertex (6) at (2.5,-1){};
    \vertex (7) at (2.5,1){};
    \path[-]

    (1) edge (3)
    (-1) edge (1)
    (1) edge (4)
    (1) edge (5)
    (5) edge (6)
    (5) edge (7)
    ;
\end{tikzpicture}
&
\begin{tikzpicture}[scale=0.7]
    \vertex (1) at (0,0){};
    \vertex (2) at (1,0){$v$};
    \vertex (3) at (0.5,-1){};
    \vertex (4) at (0.5,1){};
    \vertex (6) at (2,-1){};
    \vertex (7) at (2,1){};
   
    \path[-]
    (1) edge (2)
    (1) edge (3)
    (1) edge (4)
    (2) edge (6)
    (2) edge (7)
    (2) edge (4)
    (2) edge (3)
    (3) edge (4)
    (6) edge (7)
    ;
\end{tikzpicture}
\\
\end{tabular}
\caption{$T(1,3,2),L(T(1,3,2))$}
\label{fig:T(1,2,3)}
\end{figure}
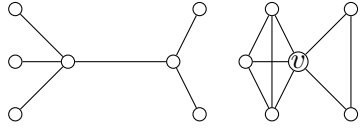
\end{ex}

The next lemma is immediate from Proposition \ref{thm:cutpoint}. 
\begin{lemma}\label{lem:linetree}
Let $X$ be a tree. Then $a(L(X))=1$ if and only if $X$ is a $T(1,s,t).$
\end{lemma}

Let $X$ be a $T(k,s,t)$ with $n$ vertices, where $k\neq1,n\geq 7$, then $a(L(X))<0.49.$  
 The line graph of $T(k,s,t)$ is a graph obtained from $T(k-1,s,t)$ by placing all possible edges among the pendant vertices attached at vertex $1$ and $k.$ We apply Theorem \ref{thm:das} to $X$ to find the Laplacian eigenvalues of $L(X).$ Let $V(X_1^\prime)$ and $V(X_2^\prime)$ contains vertices whose neighborhood is $1$ and $k,$ respectively. Consider $X^{+}=L(X)$, so $X^{+}[X_1^{\prime}]=K_s$ and $X^{+}[X_2^\prime]=K_t.$ The Laplacian eigenvalues of $K_b$ are 0 and $b$. By Theorem \ref{thm:das}, those eigenvalues of $X$ which are equal to $1$ ($s+t-2$ in number), $s-1$ and $t-1$ number of eigenvalues are incremented by $s$ and $t$, respectively. The rest of the eigenvalues are the same as $\L(X)$. This implies $a(L(X))=a(T(k-1,s,t))$ and, it is known that $a(T(k-1,s,t))<0.49.$ 
\par
\begin{theorem}\label{thm:T(1,s,t)}
 Let $X\in\S_m.$ Then $a(\be(X))=m-1$ if and only if $X$ is a $T(1,s,t)$, where $s,t\geq 2$.
\end{theorem}
\begin{proof}
Suppose that $a(\be(X))=m-1$. We know $$a(\be(X))=min\{(m-1)a(L(X)),q_{min}^{m-1}(L(X))\}.$$ Since $a(L(X))\leq 1$, this implies that $a(L(X))=1$ and $q_{min}^{m-1}(L(X))\geq m-1.$ By Lemma \ref{lem:linetree}, $X$ is a $T(1,s,t).$ To prove $s,t\geq 2$, it is sufficient to prove that if $s=1$ and $t\geq 1$, then $q_{min}^{m-1}L(X)<m-1.$
\par
$Case1:$ Let $s=t=1.$ The characterstic polynomial for $Q_{m-1}(P_3)$ is $\phi_\lambda=\lambda^3-4(m-1)\lambda^2+(5(m-1)^2-2)\lambda-2m(m-1)(m-2)=0.$ Its roots are $m-1,\frac{3(m-1)\pm\sqrt{m^2-2m+9}}{2}.$ It can be easily observed that, $\frac{3(m-1)-\sqrt{(m-1)^2+8}}{2}<m-1.$
\par
$Case2:$ Let $s=1$ and $t>1$, then we show that $q_{min}^{m-1}(L(X))<m-1.$ To prove this we choose the principal submatrix $\begin{bmatrix}
(m-1)(t+1) & 1\\
1 & (m-1)
 \end{bmatrix}$ of 
 \begin{equation*}\label{eqn:matrixQ}
Q_{m-1}=\begin{bmatrix}
 (m-1)(1+t) & 1 & \bold{1_{1\times t}}\\
 1 & m-1 & \bold{0_{1\times (n-3)}}\\
 \bold{1_{t\times1}} & \bold{0_{(n-3)\times1}} & \J_t+((m-1)t-1)\I_{t}
 \end{bmatrix},
 \end{equation*}
 where $\bold{1_{m\times n}}$ denotes the $m\times n$ matrix with all entries equal to $1.$
 Its eigenvalues are $$\frac{(m-1)(t+2)\pm \sqrt{t^2(m-1)^2+4}}{2}.$$ By Cauchy's interlacing theorem, $q_{min}^{m-1}(L(X))\leq \frac{(m-1)(t+2)- \sqrt{t^2(m-1)^2+4}}{2}
<\frac{(t+2)(m-1)-t(m-1)}{2}=m-1.$ 
\par
Conversely, suppose that $X$ is a $T(1,s,t)$, where $s,t\geq 2.$ We first show $a(\beta_2(X))=1.$ The eigenvalues of $\L(\beta_2(X))$ is given by the union of spectra of $\L(L(X))$ and $Q_1(L(X)).$ By Lemma \ref{lem:linetree}, $a(L(X))=1$. The matrix $Q_1$ has the following form
\begin{equation*}\label{eqn:structureQ}
 Q_1=\begin{bmatrix}
 s+t & \bold{1_{1\times s}} & \bold{1_{1\times t}}\\
 \bold{1_{s\times1}} & U_1 & \bold{0_{s\times (n-s-2)}}\\
 \bold{1_{t\times1}} & \bold{0_{t\times (n-t-2)}} & V_1
 \end{bmatrix},\end{equation*}
where $U_1=\J_s+(s-1)\I_s$ and $V_1=\J_t+(t-1)\I_t.$ Let $v_i,$ where $1\leq i\leq s-1$ be an eigenvector corresponding to eigenvalue $s-1$ of $U_1.$ A simple check shows that $Q_1V_i=(s-1)V_i$, where $V_i =\begin{pmatrix} 0\\  v_i\\ \bold0_{t\times 1}
\end{pmatrix}.$ Similarly, it can be seen that $t-1$ is an eigenvalue of $Q_1$ with multiplicity $t-1.$ Let $s\geq t$, for $v=(t-s,\underbrace{1,\ldots,1}_{s},\underbrace{-1,\ldots,-1}_{t})$, note that $Q_1v=(s+t-1)v$. Suppose that $\lambda_t$ and $\lambda_t^\prime$ are the remaining eigenvalues. Thus we have Equation \ref{eqn:trQ} 
\begin{equation}\label{eqn:trQ}
    \lambda_t+\lambda_t^\prime=2(s+t)-1.
\end{equation}
Since $Tr(Q_1^2)=\sum_{i=1}^{n}\lambda_i^2$, we obtain Equation \ref{eqn:lambda12}
\begin{equation}\label{eqn:lambda12}
    \lambda_t\lambda_t^\prime=-2(s+t)+4st.
\end{equation}
From Equation \ref{eqn:trQ} and \ref{eqn:lambda12}, we have
$\lambda_t^2-(2(s+t)-1)\lambda_t+(-2s-2t+4st)=0.$
Hence, $\lambda_t=(s+t)-\frac{1}{2}\pm \frac{\sqrt{4(t-s)^2+4(s+t)+1}}{2}.$ 
For $s,t\geq 2$, all the eigenvalues of $Q_1$ are greater than or equal to $1.$ Therefore, $a(\beta_2(X))=1$. 
\par 
 Now we prove for $m\geq3.$ By Theorem 4.3.1 of \cite{horn2012matrix},
 $$q_{min}^1L(X)+(m-2)\delta(L(X))\leq q_{min}^{m-1}(L(X)),$$ where $\delta(X)$ denotes the minimum of vertex degrees of $X.$ We have $q_{min}^1L(X)\geq 1.$ Since $\delta(L(X))\geq 2$, then $2m-3\leq q_{min}^{m-1}(L(X)).$ Thus, $q_{min}^{m-1}(L(X))\geq m-1.$ The proof is complete.
\end{proof}

Theorem \ref{thm:T(1,s,t)} will help us to find graph $X$ for which $\be(X)$ is Laplacian integral. Recall, a graph is {\em Laplacian integral} if the spectrum of its Laplacian matrix consists entirely of integers. A lot of work has been done on investigating the class of Laplacian integral graphs (see \cite{fallat2005graphs,kirkland2005completion,liu2010some,merris1994degree}). The spectrum of $\be(X)$ consists of eigenvalues of $(m-1)\L(L(X)).$ We know $0<a(L(X))\leq 1,$ therefore by Theorem \ref{thm:T(1,s,t)}, $X$ is a $T(1,s,t)$, where $s,t\geq 2.$ Theorem \ref{thm:T(1,s,t)} gives a necessary condition for $\be(X)$ to be Laplacian integral. Also,
the proof of converse part of Theorem \ref{thm:T(1,s,t)} shows that, let $X\in \S_m$, then $\beta_2(X)$ is Laplacian integral if and only if $X$ is a $T(1,s,t)$ tree, where $s,t\geq 2$ and $4(s-t)^2+4(s+t)+1$ is a perfect square. In Corollary \ref{cor:integral2}, we generalise it for $m.$
\par
\begin{cor}\label{cor:integral2}
Let $X\in \S_m.$ Then $\be(X)$ is Laplacian integral if and only if $X$ is a $T(1,s,t)$, where $s,t\geq2$ and
$\lambda^3-a\lambda^2+b\lambda-c=0$ has integral roots,
where $$a=(s+t)(2m-1)-2$$
$$b=m(m-1)s^2+m(m-1)t^2-(3m-1)s-(3m-1)t+m(3m-2)st+1$$
$$c=m(-m+1)s^2+m(-m+1)t^2+ms+mt-2m^2st+m^2(m-1)s^2t+m^2(m-1)st^2.$$
\end{cor}

\begin{proof}
We first give the spectrum of $\L(L(X)),$ where $X=T(1,s,t).$
The Laplacian matrix of $L(X)$ can be expressed in the following form
\begin{equation*}\label{eqn:structureL}
 \begin{bmatrix}
 s+t & \bold{-1_{1\times s}} & \bold{-1_{1\times t}}\\
 \bold{-1_{s\times1}} & U & \bold{0_{s\times (n-s-2)}}\\
 \bold{-1_{t\times1}} & \bold{0_{t\times (n-t-2)}} & V
 \end{bmatrix},\end{equation*}
 where $U=-\J_s+(s+1)\I_s$, $V=-\J_t+(t+1)\I_t.$
\par
It is easy to see that $1,s+1,t+1$ and $s+t-1$ are eigenvalues of $\L(L(X))$ with multiplicity $1,s-1,t-1$ and 1, respectively. The remaining eigenvalues are the eigenvalues of the following matrix \begin{equation*}\label{eqn:matrixQ2}
Q=Q_{m-1}(L(X))=\begin{bmatrix}
 (m-1)(s+t) & \bold{1_{1\times s}} & \bold{1_{1\times t}}\\
 \bold{1_{s\times 1}} & \J_s+((m-1)s-1)\I_{s}  & \bold{0_{s\times (n-s-2)}}\\
 \bold{1_{t\times1}} & \bold{0_{t\times (n-t-2)}} & \J_t+((m-1)t-1)\I_{t}
 \end{bmatrix}.
 \end{equation*}
 It can be easily seen that $(m-1)s-1$ and $(m-1)t-1$ are eigenvalues of $Q_{m-1}(L(X))$ with multiplicity $s-1$ and $t-1$, respectively. Let $\x,\y$ and $\z$ be the remaining three eigenvalues of $Q_{m-1}$. As the trace of a matrix is equal to the sum of its eigenvalues, we have Equation \ref{eqn:cube1}
 \begin{equation}\label{eqn:cube1}
     \x+\y+\z=(2m-1)(s+t)-2.
 \end{equation}
 Since $Tr(Q^2)=Tr(A^2)+(m-1)^2Tr(\D^2)=s(s+1)+t(t+1)+(m-1)^2((s+t)^2+s^3+t^3)$ and $Tr(Q^2)=\sum_{i=1}^{3}\lambda_i^{2}+((m-1)s-1)^2(s-1)+((m-1)t-1)^2(t-1)$, we get Equation \ref{eqn:cubesq} 
 \begin{equation}\label{eqn:cubesq}
     \sum_{i=1}^{3}\lambda_i^{2}=(1+2(m-1)^2+2(m-1))s^2+(1+2(m-1)^2+2(m-1))t^2-2(m-1)s-2(m-1)t+2(m-1)^2st+2.
 \end{equation}
Now using Equation \ref{eqn:cubesq} and the algebraic identity of $(\x+\y+\z)^2$, we deduce Equation \ref{eqn:cube2}
 \begin{equation}\label{eqn:cube2}
     \x\y+\y\z+\z\x=m(m-1)s^2+m(m-1)t^2-(3m-1)s-(3m-1)t+m(3m-2)st+1.
 \end{equation}
 
A simple computation reveals that $Tr(Q^3)=Tr(A^3)+(m-1)^3Tr(\D^3)+3(m-1)Tr(\D^2)$.
Furthermore, $Tr(Q^3)$ is equal to $\sum_{i=1}^3\lambda_i^3 + ((m-1)s-1))^3(s-1)+((m-1)t-1))^3(t-1)$. Substitute Equation \ref{eqn:cube1} and \ref{eqn:cube2} into the algebraic identity $\sum_{i=1}^3\lambda_i^3-3\x\y\z=(\x+\y+\z)((\x+\y+\z)^2-3(\x\y+\y\z+\z\x))$ and obtain Equation \ref{eqn:cube3}
 \begin{equation}\label{eqn:cube3}
     \x\y\z=m(-m+1)s^2+m(-m+1)t^2+ms+mt-2m^2st+m^2(m-1)s^2t+m^2(m-1)st^2.
 \end{equation}
 
 From Equation \ref{eqn:cube1}, \ref{eqn:cube2} and \ref{eqn:cube3}, the proof is complete.
\end{proof}
\par
In the next result, all restricted graphs $X$ are provided that have $a(X\times K_m)=m-1.$ Before that, we state a couple of results and recall some terminologies from Bapat's paper \cite{bapat2012algebraic} as it will be required to prove Theorem \ref{thm:restricted}. A {\em restricted graph} is a graph with the restriction that each block can have at most two cut vertices. For example, $L(T(k,s,t))$ is a restricted graph. A {\em block graph} is a graph in which each block is complete. Let $X$ be a connected restricted graph with blocks $B_1, B_2,\ldots, B_k.$ Replace the block $B_i$ by a path $P_i=[u_i,v_i]$ on two vertices such that $B_i, B_j$ have a common vertex if and only if $P_i, P_j$ have a common vertex. The resulting graph is a tree that we call the {\em block structure} of $X.$
\par
\begin{definition}\cite{de2011smallest}
Fix $\delta\geq 2$ and consider the minimal family of graphs $P(\delta)$ satisfying the following properties:\\
(a) if $r$ is an integer and $r>\delta$, then $K_r\in P(\delta)$;\\
(b) if $H\in P(\delta)$, $r$ and $k$ are integers, $r>\delta$ and $r-2\geq k\geq 0$, then every $k$-sum of $H$ and $K_r$ belongs to $P(\delta)$. Recall, let $X_1$ and $X_2$ be vertex disjoint graphs, and let $X_1^\prime$ be a $k$-clique of $X_1$ and $X_2^\prime$ be a $k$-clique of $X_2$. A graph that is obtained by identifying $X_1$ and $X_2$ is called a $k$-sum of $X_1$ and $X_2.$
  
\end{definition}
\begin{theorem} \cite[Theorem 2.6]{de2011smallest}\label{thm:smallQ}
Suppose that $\delta\geq 2$ and $X$ is a graph that has a spanning subgraph $H\in P(\delta).$ Also, suppose that X has two vertices of degree $\delta$ with the same closed
neighborhoods. Then $q_{min}^{1}(X)=\delta-1.$
\end{theorem}

\begin{theorem}\label{thm:restricted}
Let $X$ be a connected restricted graph with blocks $K_{s_1},\ldots,K_{s_k}$, where $k\geq3$ and $X\times K_m$ is connected. Then $a(X\times K_m)=m-1$ if and only if $\delta(X)\geq 2$ and the block structure of $X$ is a star graph.
\end{theorem}
\begin{proof}
If $a(X\times K_m)= m-1$, then $a(X)=1$ and $q_{min}^{m-1}(X)\geq(m-1).$ By Proposition \ref{thm:cutpoint}, the block structure of $X$ is a star graph. Note that $X$ has exactly one cut vertex. Suppose that there exists a pendant vertex (say $v_n$), and $v_1$ is the cut vertex of $X$. The degree of $v_1$ is $x=s_1+s_2+\cdots+s_{k-1}-k+2.$ That is, $$\begin{bmatrix}
(m-1)x & 1\\
1 & (m-1)
\end{bmatrix}$$ is a principal submatrix of $Q_{m-1}.$ Its eigenvalues are $\frac{(m-1)(x+1)\pm\sqrt{(m-1)^2(x-1)^2+4}}{2}.$ By Cauchy's interlacing theorem, $q_{min}^{m-1}(X)<m-1$, which is a contradiction.
\par
 Conversely, suppose that $\delta(X)\geq2$ and the block structure of $X$ is a star graph. By Proposition \ref{thm:cutpoint}, $a(X)=1.$ We show that $q_{min}^{m-1}(X)>m-1$. By hypothesis, $X$ is a graph obtained by attaching at least three complete graphs at one vertex. Graph $X$ satisfies the hypothesis of Theorem \ref{thm:smallQ}, hence $q_{min}^1(X)=\delta(X)-1$. By Theorem 4.3.1 of \cite{horn2012matrix}, $q_{min}^1(X)+(m-2)\delta\leq q_{min}^{m-1}(X)$, which proves the claim.
\end{proof}
\par

\section{Algebraic connectivity of some classes of graphs}\label{sec:3}
In this section, we describe the Laplacian spectrum and the algebraic connectivity of some classes of graphs. We start by defining some terms that will be used later. The graph $W(\eta,\mu)$ is a restricted graph. It is formed by attaching $\n$ copies of $K_{\u}$ in a common vertex,  where $\n\geq 2,\u\geq 3$. It is also known as {\em windmill graph}. {\em Friendship graphs}, $W(\n,3)$ are a special case of a windmill graph. The graph $W^\prime(\n,\u)$ is obtained by attaching $K_{\u}$ to each vertex of $K_{\n}.$ It can be seen that $W^\prime(\n,\u)$ is the line graph of $T(n,\n;\u-1,\u-1,\ldots,\u-1)$. The tree $T(n,k;x_1,x_2,\ldots,x_k)$ is obtained from the star graphs $K_{1,k},K_{1,x_1},\ldots,K_{1,x_k}$, by identifying the pendant vertices of $K_{1,k}$ and the centers of $K_{1,x_1}, \ldots,
K_{1,x_k}$ respectively, where $x_1\geq x_2\geq x_3\geq \cdots\geq x_k\geq 0$, $x_1\geq x_2>0$, $k\geq 2$ and $1+k+x_1+x_2+\cdots+x_k=n$ (see Figure \ref{fig:diameter4}). In Figure \ref{fig:T(2,2,3)1}, we give an example of $W(2,3)$ and $W^\prime(3,3).$ The {\em $k$-book graph}, denoted by $\mathbb{B}_k$, is defined as the graph cartesian product of $K_{1,k}$ with $K_2$. Recall that the {\em cartesian product} $X\Box Y$ of graphs  $X$ and $Y$ is a graph such that the vertex set of $X\Box Y$ is the cartesian product $V(X_1)\times V(X_2)$ and two vertices $(x_1,x_2)$ and $(x_1^\prime,x_2^\prime)$ are adjacent in $X\Box Y$ if and only if either
$x_1=x_1^\prime$ and $x_2$ is adjacent to $x_2^\prime$ in $Y$, or $x_2=x_2^\prime$ and $x_1$ is adjacent to $x_1^\prime$ in $X.$

\begin{figure}[h]
\centering
\begin{tikzpicture}[scale=0.8]
    \vertex (1) at (1,1)[label=above:$v$]{};
    \vertex (2) at (0,0)[label=left:$v_{1}$]{};
    \vertex (3) at (3,0)[label=right:$v_{k}$]{};
    \vertex (4) at (1,0)[label=left:$v_{2}$]{};
    \draw(1.5,0)node[right]{$\dots$};
 \vertex (5) at (-1,-1){};
    \vertex (6) at (-2.5,-1){};
    \draw(-2,-1)node[right]{$\dots$};
    \vertex (7) at (0,-1){};   
    \draw(0.5,-1)node[right]{$\dots$};
     \vertex (8) at (1.5,-1){}; 
   \vertex (9) at (2.5,-1){};  
   \vertex (10) at (4,-1){}; 
   \draw(3,-1)node[right]{$\dots$};
 \draw [thick,black,decorate,decoration={brace,amplitude=10pt,mirror},xshift=0.4pt,yshift=-6.5pt](-2.5,-1) -- (-1,-1) node[black,midway,yshift=-0.6cm] {\footnotesize $x_1$}; 
 \draw [thick,black,decorate,decoration={brace,amplitude=10pt,mirror},xshift=0.4pt,yshift=-6.5pt](0,-1) -- (1.5,-1) node[black,midway,yshift=-0.6cm] {\footnotesize $x_2$}; 
 \draw [thick,black,decorate,decoration={brace,amplitude=10pt,mirror},xshift=0.4pt,yshift=-6.5pt](2.5,-1) -- (4,-1) node[black,midway,yshift=-0.6cm] {\footnotesize $x_k$}; 
    \path[-]
    (1) edge (2)
    (1) edge (3)
    (1) edge (4)
    (2) edge (5)
    (2) edge (6)
    (4) edge (7)
    (4) edge (8)
    (3) edge (10)
    (3) edge (9)
     ;   
    \end{tikzpicture}
\caption{}
\label{fig:diameter4}
\end{figure}
\par
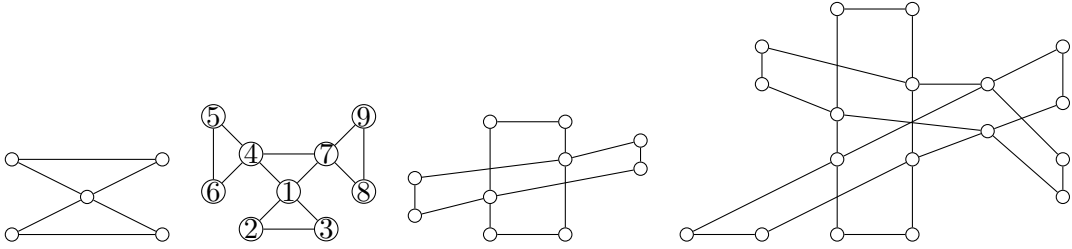
\begin{figure}[h]
\centering
\begin{tabular}{cccc}
\begin{tikzpicture}[yscale=0.5]
    \vertex (1) at (2,0){};
    \vertex (2) at (1,-1) {};
    \vertex (3) at (3,-1) {};
    \vertex (4) at (3,1){};
    \vertex (5) at (1,1) {};
    
    \path[-]
    (1) edge (2)
    (1) edge (3)
    (1) edge (4)
    (1) edge (5)
    (2) edge (3)
    (4) edge (5)
    ;
\end{tikzpicture}
&
     \begin{tikzpicture}[scale=0.5]
    \vertex (0) at (3,0){1};
    \vertex (1) at (4,1) {7};
    \vertex (2) at (2,1){4};
    \vertex (3) at (1,2){5};
    \vertex (4) at (1,0) {6};
    \vertex (5) at (2,-1) {2};
    \vertex (6) at (4,-1){3};
    \vertex (7) at (5,0){8};
    \vertex (8) at (5,2) {9};
    \path[-]
    (1) edge (0)
    (2) edge (0)
    (1) edge (2)
    (1) edge (7)
    (1) edge (8)
    (7) edge (8)
    (5) edge (6)
    (0) edge (5)
    (0) edge (6)
    (2) edge (3)
    (2) edge (4)
    (3) edge (4)
    ;
\end{tikzpicture}
& 
\begin{tikzpicture}[yscale=0.5]
    \vertex (1) at (0,0){};
    \vertex (2) at (0,1) {};
    \vertex (3) at (0,3) {};
    \vertex (4) at (1,0) {};
    \vertex (5) at (1,2) {};
    \vertex (6) at (1,3) {};
    \vertex (7) at (-1,1.5) {};
    \vertex (8) at (-1,0.5) {};
    \vertex (9) at (2,2.5) {};
    \vertex (10) at (2,1.75) {};
    
    \path[-]
    (1) edge (2)
    (2) edge (3)
    (5) edge (6)
    (4) edge (5)
    (6) edge (3)
    (1) edge (4)
    (7) edge (8)
    (7) edge (5)
    (5) edge (9)
    (9) edge (10)
    (2) edge (10)
    (2) edge (8)
    ;
\end{tikzpicture}
&
\begin{tikzpicture}[yscale=0.5]
    \vertex (1) at (0,0){};
    \vertex (2) at (0,2) {};
    \vertex (3) at (0,3.2) {};
    \vertex (4) at (0,6) {};
    \vertex (5) at (1,0) {};
    \vertex (6) at (1,2) {};
    \vertex (7) at (1,4){};
    \vertex (8) at (1,6) {};
    \vertex (9) at (-1,0) {};
    \vertex (10) at (-2,0) {};
    \vertex (11) at (2,4) {};
    \vertex (12) at (3,5) {};
    \vertex (13) at (2,2.75) {};
    \vertex (14) at (3,3.5) {};
    \vertex (15) at (-1,4.) {};
    \vertex (16) at (-1,5) {};
    \vertex (17) at (3,1) {};
    \vertex (18) at (3,2) {};
    \path[-]
    (1) edge (2)
    (2) edge (3)
    (4) edge (3)
    (1) edge (5)
    (6) edge (5)
    (7) edge (6)
    (7) edge (8)
    (4) edge (8)
    (10) edge (2)
    (9) edge (6)
    (10) edge (9)
    (2) edge (11)
    (13) edge (14)
    (6) edge (13)
    (11) edge (12)
    (12) edge (14)
    (3) edge (13)
    (7) edge (11)
    (3) edge (15)
    (7) edge (16)
    (15) edge (16)
    (13) edge (17)
    (18) edge (11)
    (17) edge (18)
    ;
\end{tikzpicture}
\\
\end{tabular}
\caption{$W(2,3),W^\prime(3,3),W(2,3)\times K_2$ and $W^\prime(3,3)\times K_2$}
\label{fig:T(2,2,3)1}
\end{figure}

\begin{theorem}\label{thm:SpecW(n,u)}
The eigenvalues of the Laplacian matrix of $W(\n,\u)\times K_m$ are 
$$0,m-1,\u(m-1),(m-1)(\u\n-\n+1),(m-1)(\u-1)-1,(m-1)(\u-1)+\u-2,\lambda_1,\lambda_2\,$$
where $\lambda_1,\lambda_2$ are the roots of $\lambda^2-((m-1)(\u-1)(\n+1)+\u-2)\lambda+\n(\u-1)((m-1)((m-1)(\u-1)+\u-2)-1)$.
\end{theorem}
\begin{proof}
The spectrum of $\L(W(\n,\u)\times K_m)$ is obtained by the union of spectra of $(m-1)\L(W(\n,\u))$ and $Q_{m-1}(W(\n,\u))$. From the structure of $W(\n,\u)$, we have
\par
$\L(W(\n,\u))=\begin{bmatrix}
\n(\u-1) & -\bold{1_{1\times \n(\u-1)}}\\
-\bold{1_{\n(\u-1)\times 1}} &  \I_{\n}\otimes P
\end{bmatrix},$ where $P=-\J_{\u-1}+\u\I_{\u-1}$.\\
For $\lambda=1$, the corresponding linearly independent eigenvectors are as follows $$\begin{bmatrix}
0\\
\bold{1_{(\u-1)\times 1}}\\
-\bold{1_{(\u-1)\times 1}}\\
\bold{0_{(\u-1)\times 1}}\\
\vdots\\
\bold{0_{(\u-1)\times 1}}
\end{bmatrix}, 
\begin{bmatrix}
0\\
\bold{1_{(\u-1)\times 1}}\\
\bold{0_{(\u-1)\times 1}}\\
-\bold{1_{(\u-1)\times 1}}\\
\vdots\\
\bold{0_{(\u-1)\times 1}}
\end{bmatrix},\ldots,
\begin{bmatrix}
0\\
\bold{1_{(\u-1)\times 1}}\\
\bold{0_{(\u-1)\times 1}}\\
\bold{0_{(\u-1)\times 1}}\\
\vdots\\
-\bold{1_{(\u-1)\times 1}}\\
\end{bmatrix}
.$$ 
Clearly, $\u$ is an eigenvalue of $P,$ let $v_1$ be the corresponding eigenvector. Consider\\ $V_1=(0,v_1,\bold{0_{(\u-1)\times 1}},\ldots,\bold{0_{(\u-1)\times 1}})^T,$ then $\L(W(\n,\u))V_1=\u V_1$ and it has algebraic multiplicity $\n(\u-2).$ Let $\lambda$ be an eigenvalue of $\L(W(\n,\u))$ with the corresponding eigenvector $$V_2 = (y,\bold{1_{(\u-1)\times 1}},\bold{1_{(\u-1)\times 1}},\ldots,\bold{1_{(\u-1)\times 1}})^T.$$ Hence, we obtain equations 
$(\n(\u-1)-\lambda)y=\n(\u-1)$ and $y=1-\lambda$ from $\L(W(\n,\u))V_2=\lambda V_2.$ Solving these equations, we get $\lambda^2-(\n\u-\n+1)\lambda=0.$ The eigenvalues of $(m-1)\L(W(\n,\u))$ are $0,m-1,\u(m-1),(m-1)(\n\u-\n+1)$ with multiplicity $1,\n-1,\n(\u-2),1$ respectively.
\par
The matrix $Q_{m-1}(W(\n,\u))$ has the following structure
$$\begin{bmatrix}
\n(\u-1)(m-1) & \bold{1_{1\times \n(\u-1)}}\\
\bold{1_{\n(\u-1)\times 1}} & \I_{\n}\otimes F
\end{bmatrix},$$
where $F=\J_{\u-1}+((m-1)(\u-1)-1)\I_{\u-1}$.
The eigenvalues of $Q_{m-1}(W(\n,\u))$ can be obtained in a way similar to $\L(W(\n,\u)).$ The eigenvalues of $Q_{m-1}(W(\n,\u))$ are $(m-1)(\u-1)-1,(m-1)(\u-1)+\u-2,\lambda_1,\lambda_2$ with multiplicity $\n(\u-2),\n-1,1,1$ respectively, where $\lambda_1,\lambda_2$ are the roots of $\lambda^2-((m-1)(\u-1)(\n+1)+\u-2)\lambda+\n(\u-1)((m-1)((m-1)(\u-1)+\u-2)-1).$
\end{proof}
\par
From Theorem \ref{thm:restricted}, we can conclude that $a(W(\n,\u)\times K_m)=m-1$. In the next result, we compute $a(W^\prime(\n,\u)\times K_m).$

\begin{theorem}\label{thm:specW'(n,u)}
The algebraic connectivity of $W^\prime(\n,\u)\times K_m$, where $\n,\u\geq 3$ is equal to $$(m-1)\left(\frac{\u+\n-\sqrt{(\u+\n)^2-4\n}}{2}\right).$$
\end{theorem}
\begin{proof}
We first explain the spectrum of $\L(W^\prime(\n,\u))$ and $Q_{m-1}(W^\prime(\n,\u)).$ By the labelling defined in Figure \ref{fig:T(2,2,3)1}, $\L(W^\prime(\n,\u))=\J_{\n}\otimes G + \I_{\n}\otimes(H-G),$ where $G$ is a $\u\times \u$ matrix with $-1$ in $(1,1)$ position and zero elsewhere, and $H=-\J_\u+\u\I_{\u}+(-\n+1)G.$ The spectrum of $\L(W^\prime(\n,\u))$ is given by the spectrum of $H-G=-\J_{\u}+\u\I_{\u}-\n G$ and $H+(\n-1)G=-\J_{\u}+\u \I_{\u}$ with multiplicity $\n-1$ and 1 respectively. It is easy to see that $\u$ is an eigenvalue of $H-G$ and its linearly independent eigenvectors are $\{(0,1,-1,0,\ldots,0)^T,
(0,1,0,-1,\ldots,0)^T,\ldots,(0, 1,0,0,\ldots,-1)^T
\}$. Let $\lambda$ be an eigenvalue of $H-G$ with corresponding eigenvector $v_1=(x,1,1,\ldots,1)^T$ such that $(H-G)v_1=\lambda v_1$, and compare to obtain equations $(\n+\u-\lambda-1)x=\u-1$ and $x=1-\lambda$. Solving these equations, we get Equation \ref{eqn:wind1}
\begin{equation}\label{eqn:wind1}
\lambda^2-(\u+\n)\lambda+\n=0
\end{equation}
The eigenvalues of $H+(\n-1)G$ are $\u$ and 0 with multiplicity $\u-1$ and 1, respectively.
\par
Note that $Q_{m-1}(W^\prime(\n,\u))=\J_{\n}\otimes(-G) +\I_{\n}\otimes(W+G),$ where $W=\J_{\u}+((m-1)(\u-1)-1)\I_{\u}-(m-1)(\n-1)G$. The spectrum of $Q_{m-1}(W^\prime(\n,\u))$ is given by the spectrum of $W+G$ and $W-(\n-1)G$ with multiplicity $\n-1$ and 1, respectively. We see that $(m-1)(\u-1)-1$ is an eigenvalue of $W+G$ and $W-(\n-1)G$ with $\u-2$ dimensional linearly independent eigenvectors $$\{(0,1,-1,0,\dots,0)^T,(0,1,0,-1,\ldots,0)^T,
\ldots,(0,1,0,0,\ldots,-1)^T
\}.$$ Consider matrix equations $(W+G)v_2=\lambda^\prime v_2$ and $(W-(\n-1)G)v_3=\lambda^{\prime\prime} v_3$, where $v_2=(y,1,1,\ldots,1)^T$ and $v_3=(z,1,1,\ldots,1)^T.$ On comparing $(W+G)v_2=\lambda^\prime v_2$, we obtain equations $((m-1)(\u+\n-2)-1-\lambda^\prime)y=-(\u-1)$ and $y+(m-1)(\u-1)+\u-2=\lambda^\prime.$ Solving these equations, we get Equation \ref{eqn:wind2}
\begin{equation}\label{eqn:wind2}
(\lambda^\prime)^2-(m\u-m-2+(m-1)(\u+\n-2))\lambda^\prime+((m-1)(\u+\n-2)-1)(m\u-m-1)-\u+1=0.
\end{equation}
Similarly, we can solve $(W-(\n-1)G)v_3=\lambda^{\prime\prime} v_3$ and obtain Equation \ref{eqn:wind3}
\begin{equation}\label{eqn:wind3}
    (\lambda^{\prime\prime})^2-(m\u-m-2+(m-1)(\u+\n-2)+\n)\lambda^{\prime\prime}+((\n-1+(m-1)(\u+\n-2))(m\u-m-1)+1-\u=0.
\end{equation}

As $a(W^\prime(\n,\u))\leq m-1$, it can be seen that $(m-1)\u$ and $(m-1)(\u-1)-1$ are greater than or equal to $m-1$. Let $\lambda_1\leq \lambda_2$, $\lambda_3\leq \lambda_4$ and $\lambda_5\leq \lambda_6$, where $\{\lambda_1,\lambda_2\},\{\lambda_3,\lambda_4\}$ and $\{\lambda_5,\lambda_6\}$ are the roots of Equation \ref{eqn:wind1}, \ref{eqn:wind2} and \ref{eqn:wind3}, respectively. We first show that $(m-1)\lambda_1<m-1$. In contrast, if $\lambda_1\geq 1$, then $\u\leq 1$, which is a contradiction. Suppose that $\lambda_3<m-1$ and $\lambda_5<m-1$, after simplification we obtain Equations \ref{eqn:W'(n,u)1} and \ref{eqn:W'(n,u)2}, respectively
\begin{equation}\label{eqn:W'(n,u)1}
(\u-2)(m((\u+\n-3)(m-1)-1)-1)-1<0
\end{equation}
\begin{equation}\label{eqn:W'(n,u)2}
    (1-\u)(1-4m)+m^2(2-\u)(3-\n)+\u m((m-1)(\u-2)-2)<0.
\end{equation}
As $m\geq 2,\u\geq 3$ and $\n\geq 3$, the left-hand sides of Equations \ref{eqn:W'(n,u)1} and \ref{eqn:W'(n,u)2} are greater than or equal to zero, which is a contradiction.
\end{proof}
\begin{re}
Note that $a(W^\prime(\n,2)\times K_2)=\lambda_3$, where $\lambda_1,\lambda_3$ and $\lambda_5$ are defined in Theorem \ref{thm:specW'(n,u)}. Clearly, $\lambda_3\leq \lambda_5$ as $\lambda_5-\lambda_3=\frac{\n+\sqrt{\n^2-4\n+8}-\sqrt{4\n^2-8\n+8}}{2}\geq \frac{\n+\n-2-2\n+4}{2}\geq 1.$ If $\lambda_1-\lambda_3\leq 0$, then $4\leq 0$ which is a contradiction. 
\end{re}
As a consequence of Theorem \ref{thm:specW'(n,u)}, we obtain Corollary \ref{thm:corW'(n,u)}. In this regard, we define some terminology. A {\em rooted tree} is a tree in which one vertex has been designated as the root. The first vertex from which the tree originates is called a {\em root vertex}. For example, in Figure \ref{fig:diameter4} vertex $v$ is the root vertex. The {\em height of a vertex} in a rooted tree is the length of the longest downward path to a pendant vertex from that vertex. The vertex that has a branch from it to any other vertex is called a {\em parent vertex}. If vertex $u$ is the parent of vertex $v$, then $v$ is called a {\em child of $u$}.
\begin{table}[h]
    \centering
    \scalebox{0.75}{    \begin{tabular}{ccc}
    \begin{tabular}{cc}
    \begin{tikzpicture}[scale=0.7]
    \vertex (1) at (1,1) {};
    \vertex (2) at (0,0) {};
    \vertex (3) at (1,0) {};
    \vertex (4) at (2,0) {};
    \vertex (5) at (-1,-1) {};
    \vertex (6) at (0,-1) {};
    \vertex (7) at (1,-1) {};
    \vertex (8) at (2,-1) {};
    \vertex (9) at (3,-1) {};
    
    \path[-]
    (1) edge (2)
    (1) edge (3)
    (1) edge (4)
    (2) edge (5)
    (6) edge (2)
    (3) edge (7)
    (4) edge (8)
    (4) edge (9)
    ;
\end{tikzpicture}
&
 \begin{tikzpicture}[scale=0.7]
    \vertex (1) at (3,3) {1};
    \vertex (2) at (3,2) {7};
    \vertex (3) at (3,1) {11};
    \vertex (4) at (5,3) {4};
    \vertex (5) at (5,2) {9};
    \vertex (6) at (5,1) {14};
    \vertex (7) at (3.5,4.5){5};
    \vertex (8) at (3.5,3.5){6};
    \vertex (9) at (4.5,4.5){2};
    \vertex (10) at (4.5,3.5){3};
    \vertex (11) at (3.5,0){15};
    \vertex (12) at (3.5,-1){16};
    \vertex (13) at (4.5,0){12};
    \vertex (14) at (4.5,-1){13};
    \vertex (15) at (2,2){8};
    \vertex (16) at (6,2){10};
    \path[-]
    (1) edge (5)
    (1) edge (6)
    (2) edge (4)
    (2) edge (6)
    (3) edge (4)
    (3) edge (5)
    (1) edge (7)
    (1) edge (8)
    (4) edge (9)
    (4) edge (10)
    (10) edge (7)
    (9) edge (8)
    (3) edge (11)
    (3) edge (12)
    (6) edge (13)
    (6) edge (14)
    (14) edge (11)
    (13) edge (12)
    (2) edge (15)
    (5) edge (16)
     ;
\end{tikzpicture}
\\
 \end{tabular}
&
\begin{tabular}{|l|llllll|llll|llllll|}
\hline
&$1$&$2$&$3$&$4$&$5$&$6$&$7$&$8$&$9$&$10$&$11$&$12$&$13$&$14$&$15$&$16$\\
\hline
$1$&4 & 0 & 0 & 0 & -1 & -1 & 0 & 0 & -1 & 0 & 0 & 0 & 0 & -1 & 0 & 0\\
$2$&0 & 2 & 0 & -1 & 0 & -1 & 0 & 0 & 0 & 0 & 0 & 0 & 0 & 0 & 0 & 0 \\
$3$&0 & 0 & 2 & -1 & -1 & 0 & 0 & 0 & 0 & 0 & 0 & 0 & 0 & 0 & 0 & 0\\
$4$&0 & -1 & -1 & 4 & 0 & 0 & -1 & 0 & 0 & 0 & -1 & 0 & 0 & 0 & 0 & 0\\
$5$&-1 & 0 & -1 & 0 & 2 & 0 & 0 & 0 & 0 & 0 & 0 & 0 & 0 & 0 & 0 & 0\\
$6$&-1 & -1 & 0 & 0 & 0 & 2 & 0 & 0 & 0 & 0 & 0 & 0 & 0 & 0 & 0 & 0\\
\hline
$7$&0 & 0 & 0 & -1 & 0 & 0 & 3 & -1 & 0 & 0 & 0 & 0 & 0 & -1 & 0 & 0\\
$8$&0 & 0 & 0 & 0 & 0 & 0 & -1 & 1 & 0 & 0 & 0 & 0 & 0 & 0 & 0 & 0\\
$9$&-1 & 0 & 0 & 0 & 0 & 0 & 0 & 0 & 3 & -1 & -1 & 0 & 0 & 0 & 0 & 0\\
$10$&0 & 0 & 0 & 0 & 0 & 0 & 0 & 0 & -1 & 1 & 0 & 0 & 0 & 0 & 0 & 0\\
\hline
$11$&0 & 0 & 0 & -1 & 0 & 0 & 0 & 0 & -1 & 0 & 4 & 0 & 0 & 0 & -1 & -1\\
$12$&0 & 0 & 0 & 0 & 0 & 0 & 0 & 0 & 0 & 0 & 0 & 2 & 0 & -1 & 0 & -1\\
$13$&0 & 0 & 0 & 0 & 0 & 0 & 0 & 0 & 0 & 0 & 0 & 0 & 2 & -1 & -1 & 0\\
$14$&-1 & 0 & 0 & 0 & 0 & 0 & -1 & 0 & 0 & 0 & 0 & -1 & -1 & 4 & 0 & 0\\
$15$&0 & 0 & 0 & 0 & 0 & 0 & 0 & 0 & 0 & 0 & -1 & 0 & -1 & 0 & 2 & 0\\
$16$&0 & 0 & 0 & 0 & 0 & 0 & 0 & 0 & 0 & 0 & -1 & -1 & 0 & 0 & 0 & 2\\
\hline
\end{tabular}\\
\\
\end{tabular}}
    \caption{$X,\beta_2(X)$ and $\L(\beta_2(X))$}
    \label{tab:diameter4ex}
\end{table}

\begin{cor}\label{thm:corW'(n,u)}
\begin{enumerate}
    \item \label{cor:child1}
    Let $T=T(n,\n;x_1,x_2,x_3,\ldots,x_{\n})$ such that if there exist atleast two $x_i$ that are equal (say $\u\geq 2$), where $1\leq i \leq \n.$ Then
$a(\be(T))\leq a(W^\prime(\n,\u+1)).$
\item \label{cor:child2}
Let $T$ be a rooted tree and, at some height $h\geq 2$ $\exists$ atleast two vertices $v_i,v_j$ of the same parent vertex that has $\n\geq 3$ children, such that $|C(v_i)|=|C(v_j)|=\u$, where $|C(v_i)|$ denotes the number of children of parent vertex $v_i$, $C(v_i)$ and $C(v_j)$ are pendant vertices of $T$. Then $a(\be(T))\leq a(W^\prime(\n,\u+1)).$
\end{enumerate}
\end{cor}
\begin{proof}
{\em Proof of part \ref{cor:child1}}. The structure of $L(T)$ is obtained by attaching $K_{x_i}$, where $1\leq i\leq \n$ at exactly one vertex of $K_{\n}.$ By hypothesis, it can be seen that there are at least two vertices in $K_{\n}$ to which $K_{\u}$ is attached. Label the vertices of one copy of $K_u$ by $\{1,2,\ldots,\u\}$ and another by $\{\u+1,\u+2,\ldots,2\u\}$, where $1$ and $\u+1$ are the vertices of $K_{\n}.$ Then observe that $\L(L(T))
=\begin{bmatrix}
H & G & C\\
G & H & C\\
C^T & C^T & D
\end{bmatrix},$ where $G$ is a $(\u+1)\times (\u+1)$ matrix with $-1$ in $(1,1)$ position and zero elsewhere, and $H=-\J_{\u+1}+(\u+1)\I_{\u+1}+(-\n+1)G.$ Let $\lambda$ be an eigenvalue of $H-G$ afforded by the eigenvector $v.$ Consider $V=(v,-v,0)^T$, then $\L(L(T))V=\lambda V.$ Now the result follows by Theorem \ref{thm:specW'(n,u)},.

{\em Proof of part 2}. It is immediate from Part \ref{cor:child1}.
\end{proof}

Next, we provide a bound for the algebraic connectivity of $L(\mathbb{B}_k)\times K_m.$ Let $H$ be a $n\times n$ matrix partitioned as $H=\begin{bmatrix}
A_{11} & A_{12}\\
A_{21} & A_{22}
\end{bmatrix},$ where $A_{11}$ and $A_{22}$ are square matrices. If $A_{11}$ is non-singular, then the Schur complement of $A_{11}$ in $H$ is defined to be the matrix $A_{22}-A_{21}A_{11}^{-1}A_{12}.$ For Schur complement, we have
$det(H)=det(A_{11})det(A_{22}-A_{21}A_{11}^{-1}A_{12}).$ 
Similarly, if $A_{22}$ is non-singular, then the Schur complement of $A_{22}$ in $H$ is defined as the matrix $A_{11}-A_{12}A_{22}^{-1}A_{21}$, and we can obtain $det(H)=det(A_{22})det(A_{11}-A_{12}A_{22}^{-1}A_{21}).$

\begin{ex}\label{ex:book}
In Figure \ref{fig:book}, we demonstrate the edge labeling process used in Theorem \ref{thm:book}. 
\begin{figure}[h]
\centering
\begin{tikzpicture}[xscale=0.9,yscale=0.6]
    \draw
    (0,0) node[circle,black, draw](1){}
    (0,3) node[circle,black, draw] (3){}
    (1,0) node[circle,black, draw](4) {}
    (1,3) node[circle,black, draw] (6){}
    (0,1.5) node[circle,black, draw](7){}
    (2,1.5) node[circle,black, draw](8){}
    (2,0)node[circle,black, draw](9) {}
    (4,0) node[circle,black, draw](10){};
     
\draw[-] (1) -- node[left] {$e_3$} (7);
\draw[-] (4) -- node[right] {$e_6$} (8); \draw[-] (1) -- node[below] {$e_9$} (4); \draw[-] (8) -- node[above] {$e_1$} (7);  
\draw[-] (9) -- node[right] {$e_4$} (7);
\draw[-] (8) -- node[above] {$e_7$} (10);
\draw[-] (9) -- node[below] {$e_{10}$} (10);   
\draw[-] (3) -- node[left] {$e_2$} (7); \draw[-] (6) -- node[right] {$e_5$} (8);
\draw[-] (3) -- node[above] {$e_8$} (6);
\end{tikzpicture}
\caption{$\mathbb{B}_3$}
\label{fig:book}
\end{figure}
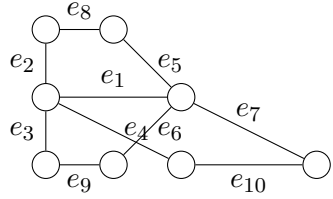
\end{ex}

\begin{theorem}\label{thm:book}
Let $X=\mathbb{B}_k$, where $k\geq 3$. Then $a(L(X))= \frac{4+k- \sqrt{k^2+8}}{2}.$
\end{theorem}

\begin{proof}
From the labeling defined in Example \ref{ex:book},
$$\L(L(X))=\begin{bmatrix}
2k & -\bold{1_{1\times 2k}} & \bold{0_{1\times k}}\\
-\bold{1_{2k\times 1}}& S & Z\\
 \bold{0_{k\times 1}}& Z^T & 2\I_k
\end{bmatrix},$$
where $S=
\I_2\otimes (-\J_k+(k+2)\I_k)$
and $Z^T=\begin{bmatrix}
-\I_{k}& -\I_{k} \end{bmatrix}.$
The characterstic polynomial of $\L(L(X))$ is 
$$det(\L(L(X))-\lambda \I_{3k+1})=det\left(\left[\begin{array}{cc|c}
2k-\lambda & -\bold{1_{1\times 2k}}  & \bold{0_{1\times k}}\\
-\bold{1_{2k\times 1}}& S-\lambda\I_{2k}& Z\\
\hline
\bold{0_{k\times 1}} & Z^T & (2-\lambda)\I_k
\end{array}
\right]\right).$$ By the Schur complement formula,
\begin{align*}
det(\L(L(X))-\lambda \I_{3k+1})&=(2-\lambda)^k det\left(\begin{bmatrix}
2k-\lambda & -\bold{1_{1\times 2k}}\\
-\bold{1_{2k\times 1}} & S-\lambda\I_{2k}
\end{bmatrix}-\begin{bmatrix}
0\\
Z
\end{bmatrix}((2-\lambda)\I_k)^{-1}\begin{bmatrix}
0 & Z^T\\
\end{bmatrix}\right)\\
&=(2-\lambda)^k det\left(\begin{bmatrix}
2k-\lambda & -\bold{1_{1\times 2k}}\\
-\bold{1_{2k\times 1}} & S-\lambda\I_{2k}
\end{bmatrix}-\frac{1}{2-\lambda}\begin{bmatrix}
0 & 0\\
0 & \J_2\otimes\I_k
\end{bmatrix}\right)\\
=&(2-\lambda)^k det\left(\begin{bmatrix}
2k-\lambda & -\bold{1_{1\times 2k}}\\
-\bold{1_{2k\times 1}} & S-\lambda\I_{2k}-\frac{1}{2-\lambda}(\J_2\otimes\I_k)
\end{bmatrix}\right).
\end{align*}
We again apply the Schur complement formula to the last expression, 
$$det(\L(L(X))-\lambda \I_{3k+1})=(2-\lambda)^k(2k-\lambda)det\left(S-\lambda\I_{2k}-\frac{1}{2-\lambda}(\J_2\otimes\I_k)-\frac{1}{2k-\lambda}\J_{2k}\right).$$
It can be verified that $S-\lambda\I_{2k}-\frac{1}{2-\lambda}(\J_2\otimes\I_k)-\frac{1}{2k-\lambda}\J_{2k}=\begin{bmatrix}
A & B\\
B & A
\end{bmatrix}$, where $$A=-\left(\frac{1}{2k-\lambda}+1\right)\J_k+\left(k+2-\lambda-\frac{1}{2-\lambda}\right)\I_k,B=-\frac{1}{2k-\lambda}\J_k-\frac{1}{2-\lambda}\I_k.$$
The spectrum of $\begin{bmatrix} 
A &B\\
B & A 
\end{bmatrix}$ is given by the union of spectra of $A+B$ and $A-B$ (See \cite{davis2013circulant}). We have $A+B=-(\frac{2}{2k-\lambda}+1)\J_k+(k+2-\lambda-\frac{2}{2-\lambda})\I_k$ and $A-B=-\J_k+(k+2-\lambda)\I_k.$ Therefore, $$det(\L(L(X))-\lambda \I_{3k+1})=(k+2-\lambda)^{k-1}(\lambda-2)\lambda(-\lambda^2+2(k+2)\lambda-(6k+2))(\lambda^2-(4+k)\lambda+2(k+1))^{k-1}.$$ 
Thus, the eigenvalues of $\L(L(X))$ are $0,k+2,2,\frac{4+k\pm \sqrt{k^2+8}}{2}$ and $k+2\pm \sqrt{k^2-2k+2}$ with multiplicity $1,k-1,1,k-1$ and 1, respectively. 
\par
Now, we show that $a(L(X))=\frac{4+k-\sqrt{k^2+8}}{2}.$ If $\frac{4+k- \sqrt{k^2+8}}{2}\geq 2$ and $\frac{4+k- \sqrt{k^2+8}}{2}\geq k+2- \sqrt{k^2-2k+2}$, then $0\geq 8$ and $2k^2(k-1)\leq 0$ respectively, which is a contradiction. The result holds.
\end{proof}
From the above theorem, we can conclude that $a(L(X)\times K_m)\leq (m-1)\frac{4+k- \sqrt{k^2+8}}{2}$.

\section{Acknowledgement} 
The authors thank the handling editor and the anonymous reviewers for their valuable comments and suggestions.

\begin{table}[h]
    \centering
    \scalebox{0.9}{
    \begin{tabular}{|c|c|c|c|c|c|c|c|c|}
    \hline
    n & $X$ & $a(X)$ & $\beta_2(X)$ & $\beta_3(X)$ & $\beta_4(X)$ & $\beta_5(X)$ & $\beta_6(X)$ & $\beta_7(X)$\\
    \hline
        $5$ & 
         \begin{tikzpicture}
    \vertex (1) at (0,0) {};
    \vertex (2) at (1,0) {};
    \vertex (3) at (2,0) {};
    \vertex (4) at (3,0.5) {};
    \vertex (5) at (3,-0.5) {};
    ;
    \path[-]
    (1) edge (2)
    (2) edge (3)
    (3) edge (4)
    (3) edge (5)
    ;
\end{tikzpicture}
& 0.519 & 0.43 & 1.72 & 2.82 & 3.87 & 4.89 & 5.91 \\
\hline
%%%%%%%
 $6$ & 
         \begin{tikzpicture}

\vertex (6) at (-1,0) {};
    \vertex (1) at (0,0) {};
    \vertex (2) at (1,0) {};
    \vertex (3) at (2,0) {};
    \vertex (4) at (3,0.5) {};
    \vertex (5) at (3,-0.5) {};
    ;
    \path[-]
    (1) edge (2)
    (2) edge (3)
    (3) edge (4)
    (3) edge (5)
    (6) edge (1)
    ;
\end{tikzpicture}
& $0.325$ & 0.224 & 1.037 & 1.556 &2.075 &2.59 &3.112 \\
\hline
%%%%%
6 &
\begin{tikzpicture}    
    \vertex (1) at (1,1) {};
    \vertex (2) at (2,1) {};
    \vertex (3) at (3,2) {};
    \vertex (4) at (3,0) {};
    \vertex (5) at (4,2) {};
    \vertex (6) at (4,0) {};
    \path[-]
    (1) edge (2)
    (2) edge (3)
    (2) edge (4)
    (3) edge (5)
    (4) edge (6)
     ;
     \end{tikzpicture}
     & 0.381 & 0.381 & 1.39 & 2.09& 2.78 & 3.486 & 4.18
\\
\hline
%%%%
$6$ & 
    \begin{tikzpicture}
    \vertex (1) at (0,0) {};
    \vertex (2) at (1,0) {};
    \vertex (3) at (2,0) {};
    \vertex (4) at (3,0.5) {};
    \vertex (5) at (3,-0.5){};
    \vertex (6) at (3,0) {};
    ;
    \path[-]
    (1) edge (2)
    (2) edge (3)
    (3) edge (4)
    (3) edge (5)
    (3) edge (6)
    ;
\end{tikzpicture}
& $0.486$ & $0.627$ & 1.824 & 2.88 & 3.91 & 4.93 & 5.94\\
\hline
%%%%%%
 6 &
\begin{tikzpicture}[scale=0.7]
    \vertex (1) at (1,1) {};
    \vertex (2) at (3,1) {};
    \vertex (3) at (4,2) {};
    \vertex (4) at (4,0) {};
    \vertex (5) at (0,2) {};
    \vertex (6) at (0,0) {};
    
    \path[-]
    (1) edge (5)
    (1) edge (6)
    (1) edge (2)
    (2) edge (3)
    (4) edge (2)
     ;
\end{tikzpicture}
& 0.438 & 1 & 2 & 3 &4 &5&6
\\
\hline
%%%%%
$7$ & 
         \begin{tikzpicture}
\vertex (7) at (-2,0) {};
\vertex (6) at (-1,0) {};
    \vertex (1) at (0,0) {};
    \vertex (2) at (1,0) {};
    \vertex (3) at (2,0) {};
    \vertex (4) at (3,0.5) {};
    \vertex (5) at (3,-0.5) {};
    ;
    \path[-]
    (1) edge (2)
    (2) edge (3)
    (3) edge (4)
    (3) edge (5)
    (6) edge (1)
    (6) edge (7)
    ;
\end{tikzpicture}
& 0.225 & $0.13$ & 0.649 & 0.974 & 1.299 & 1.62 &1.944\\
\hline
%%%%%%
$7$ & 
         \begin{tikzpicture}
    \vertex (1) at (4,0.5){};
    \vertex (2) at (1,0){};
    \vertex (3) at (2,0){};
    \vertex (4) at (3,0.5){};
    \vertex (5) at (3,-0.5){};
    \vertex (6) at (4,-0.5){};
    \vertex (7) at (5,-0.5){};
    ;
    \path[-]
    (1) edge (4)
    (2) edge (3)
    (3) edge (4)
    (3) edge (5)
    (5) edge (6)
    (6) edge (7)
    ;
\end{tikzpicture}
& 0.260 
& 0.220 & 0.826 & 1.23 & 1.652 & 2.065 &2.478\\
\hline
%%%%%
$7$ & 
\begin{tikzpicture}
    \vertex (-1) at (-1,0){};
    \vertex (1) at (0,0){};
    \vertex (2) at (1,0){};
    \vertex (3) at (2,0){};
    \vertex (4) at (3,0.5){};
    \vertex (5) at (3,-0.5){};
    \vertex (6) at (3,0){};
    ;
    \path[-]
    (1) edge (2)
    (2) edge (3)
    (3) edge (4)
    (3) edge (5)
    (3) edge (6)
    (1) edge (-1)
    ;
\end{tikzpicture}
& 0.296 & 0.28 & 0.9717 & 1.45&1.943 &2.42 &2.91\\
\hline
%%%%%
7 &
\begin{tikzpicture}    
    \vertex (1) at (3,1) {};
    \vertex (2) at (2,1) {};
    \vertex (3) at (3,2) {};
    \vertex (4) at (3,0) {};
    \vertex (5) at (4,2) {};
    \vertex (6) at (4,0) {};
    \vertex (7) at (4,1) {};

    \path[-]
    (1) edge (2)
    (2) edge (3)
    (2) edge (4)
    (3) edge (5)
    (4) edge (6)
    (1) edge (7)
    
     ;
     \end{tikzpicture}
   & 0.381  & 0.381 & 1.39 & 2.09&2.075 &2.59 &3.112\\
\hline
7 &
\begin{tikzpicture}
    \vertex (1) at (1,1) {};
    \vertex (2) at (2,1) {};
    \vertex (3) at (3,1) {};
    \vertex (4) at (4,2) {};
    \vertex (5) at (4,0) {};
    \vertex (6) at (0,2) {};
    \vertex (7) at (0,0) {};
    \path[-]
    (1) edge (2)
    (2) edge (3)
    (3) edge (4)
    (3) edge (5)
    (1) edge (6)
    (1) edge (7)
    ;
\end{tikzpicture}
& 0.267&
0.43 & 0.876 & 1.31&1.72 &2.15 &2.58
\\
\hline
%%%%%
7 &
\begin{tikzpicture}[scale=0.7]
    \vertex (1) at (1,1) {};
    \vertex (2) at (3,1) {};
    \vertex (3) at (4,2) {};
    \vertex (4) at (4,0) {};
    \vertex (5) at (0,2) {};
    \vertex (6) at (0,0) {};
    \vertex (7) at (0.5,1.5) {};
    
    \path[-]
     (1) edge (6)
    (1) edge (2)
    (2) edge (3)
    (4) edge (2)
    (5) edge (7)
    (1) edge (7)
     ;
\end{tikzpicture}
& 0.322& 0.45& 1.26 & 1.89&2.52 &2.95 &3.38
\\
\hline

%%%%%%
$7$ & 
         \begin{tikzpicture}
    \vertex (1) at (4,0.5){};
    \vertex (2) at (1,0){};
    \vertex (3) at (2,0){};
    \vertex (4) at (3,0.5){};
    \vertex (5) at (3,-0.5){};
    \vertex (6) at (4,-0.5){};
    \vertex (7) at (2,-0.5){};
    ;
    \path[-]
    (1) edge (4)
    (2) edge (3)
    (3) edge (4)
    (3) edge (5)
    (5) edge (6)
    (3) edge (7)
    ;
\end{tikzpicture}
& 0.381
& 0.581 & 1.52 & 2.29 &3.055 &3.81 &4.58\\
\hline
%%%%%

$7$ & 
         \begin{tikzpicture}
    \vertex (1) at (0,0) {};
    \vertex (2) at (1,0) {};
    \vertex (3) at (2,0) {};
    \vertex (4) at (3,0.5) {};
    \vertex (5) at (3,-0.5){};
    \vertex (6) at (3,0) {};
    \vertex (7) at (2,1) {};
    ;
    \path[-]
    (1) edge (2)
    (2) edge (3)
    (3) edge (4)
    (3) edge (5)
    (3) edge (6)
    (3) edge (7)
    ;
\end{tikzpicture}
& 0.466 & 0.72 & 1.86 & 2.91 & 3.93 &4.94 &5.95\\
\hline
%%%%%%
 7  &
\begin{tikzpicture}[scale=0.7]
    \vertex (1) at (1,1) {};
    \vertex (2) at (3,1) {};
    \vertex (3) at (4,2) {};
    \vertex (4) at (4,0) {};
    \vertex (5) at (0,2) {};
    \vertex (6) at (0,0) {};
    \vertex (7) at (4,1) {};
    
    \path[-]
    (1) edge (5)
    (1) edge (6)
    (1) edge (2)
    (2) edge (3)
    (4) edge (2)
    (2) edge (7)
     ;
\end{tikzpicture}
& 0.398
& 1 & 2 & 3 &4 &5&6\\
\hline
\end{tabular}
    }
    \caption{Caption}
    \label{tab:algebraic}
\end{table}
\bibliographystyle{abbrv}
\bibliography{bibliography}
\end{document}